\documentclass[12pt]{article}
\usepackage[utf8]{inputenc}

\title{On a colored Tur\'an problem of Diwan and Mubayi}
\author{Ander Lamaison \thanks{Freie Universit\"at, Institut für Mathematik and Berlin Mathematical School, email: \texttt{lamaison@zedat.fu-berlin.de}.} \and Alp M\"uyesser\thanks{Freie Universit\"at, Institut für Mathematik and Berlin Mathematical School, email: \texttt{alp.muyesser@fu-berlin.de}.} \and Michael Tait\thanks{Villanova Univesity Department of Mathematics and Statistics, email: \texttt{michael.tait@villanova.edu}. Research is partially supported by National Science Foundation grant DMS-2011553.}}
\date{\vspace{-5ex}}
\usepackage{amsmath}
\usepackage{amssymb}
\usepackage{amsthm}
\usepackage{amsfonts}
\usepackage{ bbold }
\usepackage{graphicx}
\usepackage{caption}
\usepackage{subcaption}
\usepackage{changepage}
\usepackage{enumerate}
\usepackage{ dsfont }
\usepackage{soul}
\usepackage{hyperref}
\usepackage{tikz}
\usepackage{braket}
\usepackage[a4paper, total={6in, 8in}]{geometry}

\newcommand{\ep}{\varepsilon}

\newcommand{\mex}{\mathrm{mex} }

\theoremstyle{plain}
\newtheorem{theorem}{Theorem}[section]
\newtheorem{lemma}[theorem]{Lemma}
\newtheorem{proposition}[theorem]{Proposition}
\newtheorem{claim}{Claim}[section]

\newtheorem{definition}[theorem]{Definition}

\begin{document}
\maketitle
\begin{abstract}
    Suppose that $R$ (red) and $B$ (blue) are two graphs on the same vertex set of size $n$, and $H$ is some graph with a red-blue coloring of its edges. How large can $R$ and $B$ be if $R\cup B$ does not contain a copy of $H$? Call the largest such integer $\mex(n, H)$. This problem was introduced by Diwan and Mubayi, who conjectured that (except for a few specific exceptions) when $H$ is a complete graph on $k+1$ vertices with any coloring of its edges $\mathrm{mex}(n,H)=\mathrm{ex}(n, K_{k+1})$. This conjecture generalizes Tur\'an's theorem. 
   \par Diwan and Mubayi also asked for an analogue of Erd\H{o}s-Stone-Simonovits theorem in this context. We prove the following asymptotic characterization of the extremal threshold in terms of the chromatic number $\chi(H)$ and the \textit{reduced maximum matching number} $\mathcal{M}(H)$ of $H$. 
   $$\mex(n, H)=\left(1- \frac{1}{2(\chi(H)-1)} - \Omega\left(\frac{\mathcal{M}(H)}{\chi(H)^2}\right)\right)\frac{n^2}{2}.$$
   \par $\mathcal{M}(H)$ is, among the set of proper $\chi(H)$-colorings of $H$, the largest set of disjoint pairs of color classes where each pair is connected by edges of just a single color. The result is also proved for more than $2$ colors and is tight up to the implied constant factor.
   \par We also study $\mex(n, H)$ when $H$ is a cycle with a red-blue coloring of its edges, and we show that $\mex(n, H)\lesssim \frac{1}{2}\binom{n}{2}$, which is tight.  
\end{abstract}

\section{Introduction}
Let $G_1,\ldots, G_r$ be not necessarily distinct graphs \textit{on the same vertex set}, and let the edge set of $G_i$ be colored with color $i$. By $\bigsqcup_{i\in[r]} G_i$, we denote the $r$-edge colored multigraph formed by taking the union of the edge sets $E(G_i)$.
\begin{definition}
Let $H$ be an $r$-edge colored (multi-)graph. By $\mex(r, n, H)$ we denote the maximum integer $T$ such that there exists graphs $G_1,\cdots, G_r$ on the same set of $n$ vertices such that $|E(G_i)|\geq T$ for all $i\in[r]$ and $\bigsqcup G_i$ does not contain a copy of $H$. We define $\mex(r, H)$ to be $\lim_{n\to\infty} \frac{\mex(r,n, H)}{\binom{n}{2}}$.
\end{definition}
When $r=2$, Diwan and Mubayi \cite{DM} initiated the study of the above parameter. They were focused on the case when $H=K_{k+1}$ is a clique with an arbitrary coloring of its edges. They conjectured that when $k\geq 8$, regardless of the edge-coloring of $H$, the extremal threshold is the same as that in the colorless setting, namely, $1-1/k$ (see Tur\'an's theorem \cite{Diestel}). When $k<8$, they conjectured that the same result holds excluding some edge-colorings of $K_{k+1}$. For partial progress on this conjecture, we refer the reader to \cite{DM} and \cite{thomason} (see Section $8$). In this paper, we will be concerned with the question of determining $\mex(r, H)$ when $H$ is not necessarily complete.
\par Recall the Erd\H{o}s-Stone-Simonovits theorem, which states that for any graph $H$, $\textrm{ex}(n, H)=\left(1-\frac{1}{\chi(H)-1}\right)\binom{n}{2} + o(n^2)$. The theorem thus asymptotically characterizes the extremal threshold of any non-bipartite graph. Diwan and Mubayi asked for an analogue of the Erd\H{o}s-Stone-Simonovits theorem in the colorful setting \cite{DM}. As we don't even fully understand the behaviour of $\mex(\cdot)$ on complete graphs, it is premature to hope for a single parameter that characterizes $\mex(r, H)$ for any $r$-edge-colored $H$. Here, we aim to designate a parameter of edge-colored graphs that approximately determines their extremal threshold.
\par Note that any such bound on $\mex(r, H)$ for general $H$ must take into account the coloring associated with the edges of $H$, and therefore depend on a parameter other than just the ordinary vertex-chromatic number of $H$. As an example, observe that for any monochromatic bipartite $H$, $\textrm{mex}(r, H)=0$ whereas when $H$ is a $2$-edge path colored red-blue, $\textrm{mex}(r, H)=1/4$.
\par For a first general upper bound, consider when $r=2$, and $H$ is a red/blue edge-colored $K_{k+1}$. It is easy to see that $\mex(2, H)\leq 1 - \frac{1}{2k}$. Indeed, when $|R|,|B|> \left(1 -\frac{1}{2k}\right)\frac{n^2}{2}$, $|R\cap B|> \left(1-\frac{1}{k}\right)\frac{n^2}{2}$ and by Tur\'an's theorem, $R\cap B$ contains a $K_{k+1}$. Now, regardless of the coloring of $H$, it will be possible to embed $H$ into $R\sqcup B$. Similarly, if $H$ was any $r$-edge-colored $(k+1)$-chromatic graph, if each color class has density more than $1 - \frac{1}{r(\chi(H)-1)}$, then the intersection of all of the colors has density more than $1-\frac{1}{\chi(H)-1}$. The Erd\H{o}s-Stone theorem then tells us that a copy of a supergraph of $H$ where every edge has multiplicity $r$ can be found. Therefore we have $\mex(r, H) \leq \left(1 - \frac{1}{r(\chi(H)-1)}\right)$ for any $H$. Perhaps surprisingly, we will see in Section \ref{sec:construction} that for some $r$-edge-colored $(k+1)$-chromatic graphs, this bound is tight. Hence, any possible analogue of Erd\H{o}s-Stone-Simonovits theorem in this context can only characterize how far away from the trivial upper bound $\mex(r, H)$ is, taking into account the specific $r$-edge-coloring of $H$. 
\par We will now introduce a parameter ($\mathcal{M}(\cdot)$) that is in some sense a measure of how monochromatic a bicolored $H$ is. Using this parameter, we will provide a generalization of the celebrated Erd\H{o}s-Stone-Simonovits theorem into the colorful setting. In short, $\mathcal{M}(H)$ is, among the set of proper $\chi(H)$-colorings of $H$, the largest set of disjoint pairs of color classes where each pair is connected by edges of just a single color. We give a more precise definition in what follows.
\begin{definition}[Reduced graph] Let $G$ be an $r$-edge-colored graph. We say that an $r$-edge-colored multigraph $P$ is a reduced graph of $G$ if all of the following hold. 
\begin{enumerate}[(a)]
    \item\label{a} there exists a function $f\colon V(G)\to V(P)$ such that the existence of a c-colored edge $\{x,y\}$ in $G$ implies that there is a $c$-colored edge between $f(x)$ and $f(y)$ in $P$
    \item\label{b} $f$ induces a proper vertex-coloring of $G$, i.e. each $f^{-1}(\{v\})$ is an independent set for any $v\in V(P)$
    \item No proper induced subgraph of $P$ satisfies (\ref{a}) and (\ref{b}).  
\end{enumerate}
By $\mathcal{R}(G)$, we denote the family of all reduced graphs of $G$.
\end{definition}
\par If $P$ is a multigraph, we denote by $M(P)$ the maximum size of a matching in the subgraph of $P$ made up of the edges of multiplicity $1$. 
\begin{definition}[Reduced maximum matching]
Let $G$ be some $r$-edge-colored graph, and let 
$$\mathcal{M}(G)=\max_{\substack{P\in \mathcal{R}(G)\\ |P|=\chi(G)}} M(P).$$
denote the ``reduced maximum matching number'' of $G$. 
\end{definition}
We can now state our main theorem. 
\begin{theorem}[Multigraph Erd\H{o}s-Stone-Simonovits]\label{thm:maintheorem}
Let $G$ be some $r$-edge colored graph. Then, $$\mex(r, G)\leq 1- \frac{1}{r(\chi(G)-1)} - \frac{\mathcal{M}(G)}{9r\chi(G)^2}.$$ Further, there exist $r$-edge-colored graphs $G$ so that the bound is best possible up to the multiplicative factor $1/9$, whenever $\mathcal{M}(G)\leq \chi(G)/10$.
\end{theorem}
\par Although Theorem \ref{thm:maintheorem} is best possible in general up to the constant factors, it is not necessarily tight when $\mathcal{M}(G)$ is near $\chi(G)/2$, for example when $G$ is just a clique on $\chi(G)$ vertices. There are other natural settings in which our Theorem \ref{thm:maintheorem} does not necessarily give a tight answer, and we discuss these problems more in Section 5. 
\par We also study the parameter $\mex(G):=\mex(2, G)$ for some specific classes of graphs.  If $G$ is bicolored and bipartite, observe that if $R\cup B$ avoids $G$, $|R\cap B|=o(n^2)$. Indeed, otherwise $R\cap B$ would contain a large complete bipartite graph by the Erd\H{o}s-Stone-Simonovits theorem. Hence, here we may assume that $R\cap B = \emptyset$ without changing densities. It follows that $\mex(G)\leq 1/2$. Further, the hypothesis that $R\cap B=\emptyset$ brings us to the setting of \cite{TM}. (Indeed, if $R$ and $B$ don't meet, and they both have density at least $\alpha$, their union contains a $(1/2)$-balanced graph of density $2\alpha$.) Using the Theorem 1.1 in \cite{TM}, we can then obtain a better upper bound on $\mex(G)$ when $G$ is bipartite and ``inevitable"\footnote{For definitions of ($1/2$)-balanced and ``inevitable", we refer the reader to $\cite{TM}$}. In particular, in this case it will be that $\mex(G)\leq 1/2 - f(v(G))$ where $f$ is positive, but exponentially small in $v(G)$.
\par From the previous discussion, it follows that for any even bicolored cycle, $\mex(G)\leq 1/2$. We also show that odd bicolored cycles can also be found at this same threshold.  
\begin{theorem}\label{thm:cycle}
Let $C$ be any bicolored cycle. Then, $\mex(C)\leq 1/2$.
\end{theorem}
We include the short proof of Theorem \ref{thm:cycle} in the Discussion section, along with some open problems. In the next section, we give a construction showing that Theorem \ref{thm:maintheorem} is tight up to the factor $\frac{1}{9}$. Then in Sections \ref{sec:regularity} and \ref{sec:upper bound}, we prove the upper bound in Theorem \ref{thm:maintheorem}. 

\section{Construction}\label{sec:construction}
Our goal in this section is to demonstrate the sharpness of Theorem \ref{thm:maintheorem}. For even $r$, we will construct an $r$-edge-colored graph $H$, with chromatic number $k$, such that $\mex(r,H)\geq 1-\frac{1}{r(k-1)}$. First, assuming such an $H=H(r,k)$ for every even $r$ and $k$, let us show that Theorem \ref{thm:maintheorem} is sharp. Given an even $r$, and positive integers $m$ and $k$ such that $10m\leq k$ we will construct a graph $H'$  for which $\mathcal{M}(H')\geq m$, $\chi(H')=k$ and
$$\mex(r, H')\geq 1- \frac{1}{r(\chi(H')-1)} - O\left(\frac{\mathcal{M}(H')}{r\chi(H')^2}\right).$$
\par The construction of $H'$ is as follows. Let $H:=H(r, k-2m)$, and let $H'$ be the union of $H$ and a red $(2m)$-clique, and we add all possible edges between the clique and $H$ in red. It is clear that $\mathcal{M}(H')\geq m$ and $\chi(H')=k$. And as $H'$ contains $H$, it follows that $$\mex(r, H')\geq 1 - \frac{1}{r(k-2m-1)} = 1- \frac{1}{r(\chi(H')-1)} - O\left(\frac{\mathcal{M}(H')}{r\chi(H')^2}\right)$$
where in the last equality we used that $10m\leq k$. 

\subsection{The graph $H$}

The construction of $H$ is as follows. $V(H)$ has size $tk$, where $t$ is a large constant that will be defined later. Denote the vertices by $v_{i,j}$, where $i\in[t]$ and $j\in [k]$. For every pair $j,j'$, between the vertex sets $\{v_{i,j}\}_{i=1}^t$ and $\{v_{i,j'}\}_{i=1}^t$ we have copy of a graph $F$, which is a an $r$-colored complete bipartite graph $K_{t,t}$. We need the colors of the edges of $F$ to satisfy the following property: if the vertex set of $F$ has bipartition $X\cup Y$, then for any pair of sets $X'\subseteq X$, $Y'\subseteq Y$, each with size $\frac{t}{(rk)^k}$, the bipartite graph induced on $X',Y'$ has edges on every color in $[r]$.

\begin{claim}
If $r$ and $k$ are fixed, then for $t$ large enough there is a complete bipartite graph with $t$ vertices in each part and an $r$-coloring of the edges such if $X'\subseteq X$, $Y'\subseteq Y$, and $|X'|, |Y'| \geq \frac{t}{(rk)^k}$, the bipartite graph induced on $X',Y'$ has edges on every color in $[r]$.
\end{claim}

\begin{proof}

Let $F$ be a complete bipartite graph with $t$ vertices in each part. Color the edges independently and uniformly at random from $[r]$. By the union bound, the expected number of pairs $X',Y'$ of size $\frac{t}{(rk)^2}$ which are missing some color is at most $$r\binom{t}{\frac{t}{(rk)^k}}^2\left(\frac{r-1}{r}\right)^{\left(\frac{t}{rk}\right)^2} \leq r(erk)^{\frac{2t}{(rk)^k}} e^{\frac{-t^2}{r^3k^2}},$$ which is less than 1 for $t$ large enough.
\end{proof}

Putting a copy of $F$ between each of the $k$ parts completes the construction of $H$. Since $H$ is a complete $k$-partite graph it has chromatic number $k$.
\subsection{Constructing $G_1, \cdots, G_r$}
Next we will construct graphs $G_1,\dots, G_r$, each with $(1-\frac{1}{r(k-1)}+o(1)){n\choose 2}$ edges, such that $\bigsqcup G_i$ does not contain a colored copy of $H$. See Figure \ref{fig:test1} for a diagram of the construction in the case when $r=4$ and $k=3$. Each $G_i$ will be on the same vertex set of size $n$ where $n=(k-1)rm$, for some integer $m$ (and $k$ and $r$ fixed as before). Denote the vertices in these graphs by $w_{x,y,z}$, for $x\in[k-1]$, $y\in[r]$ and $z\in[m]$. 

It will be more convenient to define the complement $G_i^c$ of these graphs. The graph $G_r$ will be defined differently from $G_1,\cdots, G_{r-1}$. The edges in $G_r^c$ are precisely the edges of the form $w_{x,y,z}w_{x,y,z'}$ for all $1\leq z,z'\leq m$. Another way to say this is that $G_r$ is a Tur\'an graph on $(k-1)r$ parts.

To define the edge sets of $G_1,\cdots, G_{r-1}$, first let $U$ be a clique $K_r$, where the edges are properly colored with the colors in $[r-1]$ (such a coloring exists because $r$ is even).  For each $i\in[r-1]$, the edges in $G_i^c$ are precisely those of the form $w_{x,y,z}w_{x,y',z'}$, where the edge $yy'$ receives color $i$ in $U$ and for all $1\leq z, z'\leq m$. 

Each of the $G_i^c$ have $(1+o(1))\frac{1}{r(k-1)}{n\choose 2}$ edges, so $e(G_i)/\binom{n}{2} \sim \left(1-\frac{1}{r(k-1)}\right)$.

\begin{figure}
\centering
\begin{minipage}{.6\textwidth}
  \centering
  \includegraphics[width=1\linewidth]{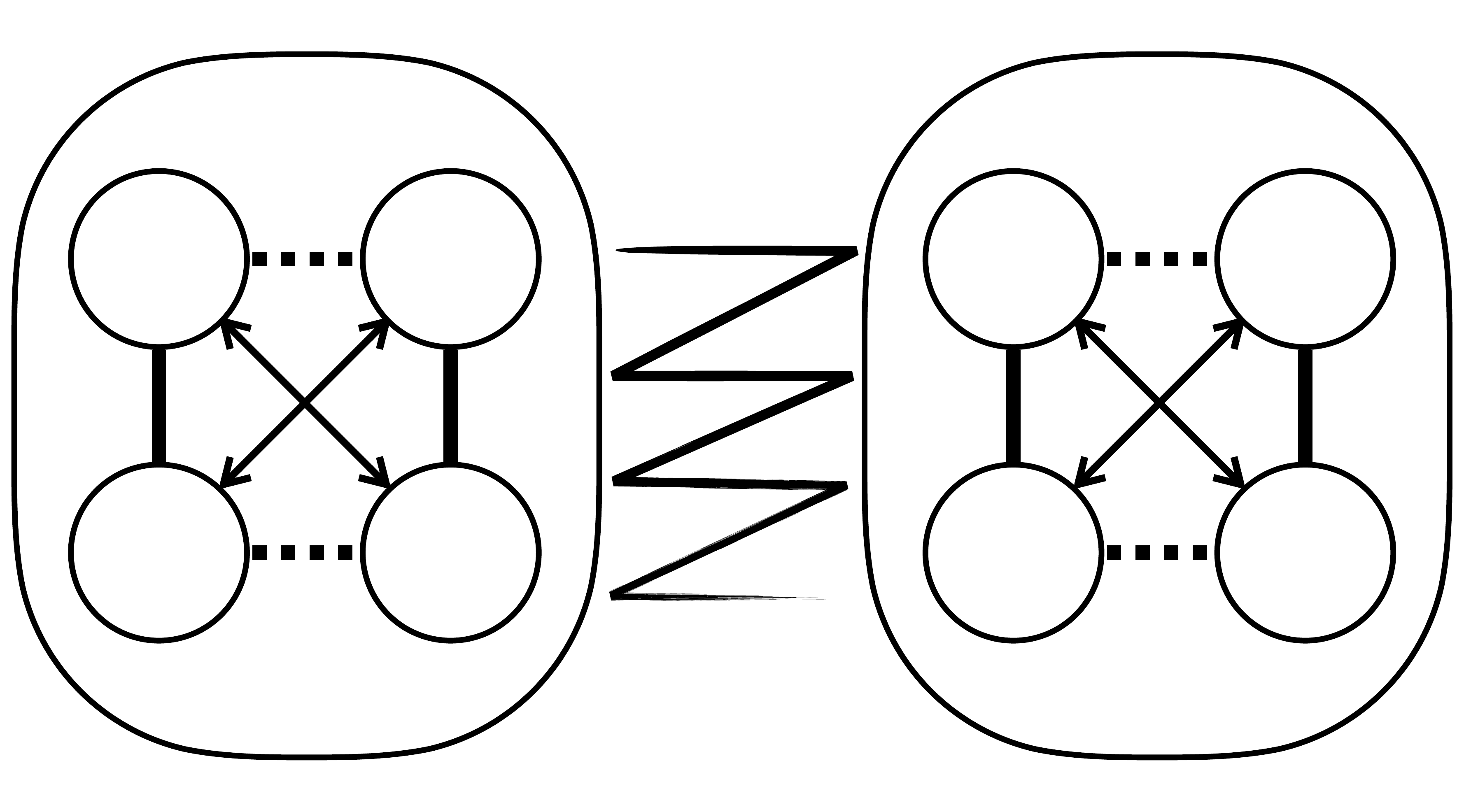}
  \captionof{figure}{An illustration of $G_1,\cdots, G_r$ when $r=4$ and $k=3$. The small circles represent subsets of $n/8$ vertices. $G_4$ is a Tur\'an graph with these circles as the parts. $G_1$ can be obtained by starting with a complete graph, and then removing the edges in the bipartite graphs corresponding to the dotted lines. $G_2$ and $G_3$ is defined similarly, replacing dotted with solid and arrowed, respectively. }
  \label{fig:test1}
\end{minipage}%
%\begin{minipage}{.55\textwidth}
%  \centering
%  \includegraphics[angle=270,origin=c, width=.5\linewidth]{target_1}
%  \captionof{figure}{Target graph}
%  \label{fig:test2}
%\end{minipage}
\end{figure}

\subsection{Showing that $\bigsqcup G_i$ is $H$-free}
We will next prove that $\bigsqcup G_i$ indeed does not contain a copy of $H$. Suppose that there is such a copy of $H$. For each $i$, consider the vector $(x_1^i, y_1^i, x_2^i, y_2^i, \dots, x_k^i, y_k^i)$, where $x_j^i$ and $y_j^i$ are the first and second coordinate of the image of $v_{i,j}$ in $\bigsqcup G_i$. By the pigeonhole principle, there is a subset $I\subseteq [t]$ of at least $\frac{t}{(rk)^k}$ values of $i$ for which the vector defined above is the same.

There are two values $j\neq j'$ such that $x_j^i=x_{j'}^i$ for all $i\in I$. If we have $y_j^i=y_{j'}^i$, then the images of all edges of the form $v_{i,j}v_{i',j'}$ for $i,i'\in I$ are in $G_r^c$. If on the other hand we have $y_j^i\neq y_{j'}^i$, then the images of all edges of the form $v_{i,j}v_{i',j'}$ for $i,i'\in I$ are in $G_q^c$ for some $q\in[r-1]$. Regardless of the case, this contradicts the fact that the bipartite graph formed by these edges contains all colors.

\section{Reduction to reduced graphs}\label{sec:regularity}
Here, we show that the extremal threshold of any graph $G$ is at most as large as the threshold of any of its reduced graphs. We now state the main result of this section.
\begin{proposition}\label{prop:erdosstone}
Let $G$ be an $r$-edge-colored graph, and let $P\in \mathcal{R(G)}$ be one of its reduced graphs. Then, $\mex(r, G)\leq \mex(r, P)$.  
\end{proposition}
The proof will use the regularity lemma and is similar to the regularity based proof of the Erd\H{o}s-Stone theorem. Before we begin, we must state the regularity lemma, starting with the necessary terminology. Let $G:=(A,B)$ be a bipartite graph with $|A|=|B|=n$. For $X\subset A$ and $Y\subset B$ define $d(X,Y) := \frac{e(X,Y)}{|X||Y|}$. We call $G$ $\ep$-regular if for all subsets $X\subseteq A$ and $Y\subseteq B$ with $|X|, |Y|\geq \ep n$ we have $|d(X,Y)-d(A,B)|\leq \ep$. 
\begin{lemma}[Szemer\'edi \cite{regularity}]\label{regularity lemma}
For any $\ep>0$ there exists an $M:=M(\ep)$ such that any graph $G$ can be partitioned into $k$ (where $\frac{1}{\ep}\leq k \leq M$) equal sized parts $(V_i)_{i\in [k]}$ and a junk set $J$ with $|J|<\ep n$ such that all but $\ep$-fraction of the pairs $(V_i,V_j)$ are $\ep$-regular.  
\end{lemma}
We also record a multi-color version of the regularity lemma we will need, which can be proved by iterating the regularity lemma multiple times (c.f. \cite{regularityDIMACS} Theorem 1.18).
\begin{lemma}\label{colored regularity lemma}
For any $\ep>0$ there exists an $M:=M(\ep)$ such that any graph $\bigsqcup_{i\in[r]}G_i$ can be partitioned into $k$ (where $\frac{1}{\ep}\leq k \leq M$) equal sized parts $(V_i)_{i\in [k]}$ and a junk set $J$ with $|J|<\ep n$ such that all but $\ep$-fraction of the pairs $(V_i,V_j)$ are $\ep$-regular in the color $i$ for all $i\in[r]$.  
\end{lemma}

Given a multigraph $H$ with an $r$-coloring of its edges, we define a {\em blow-up of $H$} with $s$ vertices in each part to be the graph where each vertex of $H$ is replaced by an independent set of size $s$ and each edge in $H$ of color $i$ is replaced by a complete bipartite graph in color $i$. We denote this blow-up by $H^{(s)}$.

In order to prove Proposition \ref{prop:erdosstone}, it suffices to show that if $d(G_i)\geq \mex(r, P) + \delta$ for all $i\in [r]$ and $n$ is sufficiently large, then $\bigsqcup G_i$ contains a blow-up of $P$ with $|G|$ vertices on each part. The fact that a sufficiently large blow-up of $P$ will contain a copy of $G$ follows from the definition of $P$ being a reduced graph of $G$.

The regularity lemma gives a streamlined method of finding such blow-ups in dense graphs. Given an $\ep$-regular partition of an $r$-colored graph $G$ where the parts have size $\ell$ and a $d>0$, we define the {\em multicolor regularity multigraph} with parameters $\ep$, $\ell$, and $d$ to be the $r$-colored multigraph where vertices are indexed by the parts $V_i$ and there is an edge in color $c$ between $V_i$ and $V_j$ if the pair $(V_i, V_j)$ is $\ep$-regular with density at least $d$ in color $c$. Applications of the regularity lemma often use an embedding lemma which says that subgraphs of a blow-up of the regularity graph $R^{(s)}$ will also be found in the original graph. The following lemma is a colorful version of such an embedding lemma. The proof is a rewriting of the uncolored version (see Lemma 7.5.2 in Diestel \cite{Diestel}). 

\begin{lemma}\label{lem:embedding}
For all $d\in (0,1]$ and $\Delta, s\geq 1$, there exist an $\ep_0$ and an $L$ such that if $G$ is an $r$-colored graph, $H$ is an $r$-colored multigraph with maximum degree $\Delta$ in each color, and $R$ is a multicolor regularity multigraph of $G$ with parameters $\ep \leq \ep_0$, $\ell \geq L$, and $d$, then 
\[
H\subset R^{(s)} \implies H\subset G 
\]
\end{lemma}

%The following lemma is integral to regularity based proofs of the Erd\H{o}s-Stone theorem. We state it in a colored setting, yet it can be proved with the same greedy embedding scheme. For example, see Lemma 7.3.2 in Diestel \cite{diestel}. If $\mathcal{C}$ is a collection of colors, by a $\mathcal{C}$-colored edge we mean a pair of vertices $\{u,w\}$ where $\mathcal{C}$ is the set of colors of edges that occur between $\{u,w\}$.

%\begin{lemma}\label{lem:embedding}
%Let $P$ be a reduced graph of an $r$-edge colored graph $H$ and let $d\in(0,1)$. Then, there exists a sufficiently small $\ep\in(0,1)$ such that the following holds. Let $\mathcal{G}:=\bigsqcup_{i\in[r]}G_i$ be some graphs on the same vertex set, and let $\bigsqcup_{i\in[v(P)]} R_i$ be a partition of the vertex set so that each part is sufficiently large and for any subset of colors $\mathcal{C}\subseteq [r]$ the following holds. 
%\begin{enumerate}[(a)]
%    \item For any pair $(R_i, R_j)$ the graph with edges that are the $\mathcal{C}$-edges between $(R_i, R_j)$ is $\ep$-regular.
 %   \item If there exists an edge that is $\mathcal{C}$-colored between vertex $i$ and $j$ in $P$, the density of the $\mathcal{C}$-edges in $(R_i, R_j)$ is at least $d$.
%\end{enumerate}
%Then, $\mathcal{G}$ contains a $v(G)$-blow-up of $P$, and thus, a copy of $H$.  
%\end{lemma}
Using Lemma \ref{lem:embedding}, it suffices to show that $P$ is a subgraph of the multicolor regularity multigraph $R$, for then $P^{(s)}$ will be a subgraph of $R^{(s)}$. Showing this will be a consequence of the regularity lemma. We would like to highlight that this proof is a standard application of the regularity method. We include the details for completeness.
\begin{proof}[Proof of Proposition \ref{prop:erdosstone}]
We let $n$ be sufficiently large, and fix $\bigsqcup_{i\in[r]} G_i$ to be some graphs on the same vertex set $[n]$ such that the density of each $G_i$ is at least $\mex(r, P) + \delta$ for some positive constant $\delta$.  Choose $\ep = \delta/16$ and $d = \delta/4$. Assume that $\ep$ is also chosen small enough that it is less than the $\ep_0$ from Lemma \ref{lem:embedding} with parameters $d$ and $\Delta(P)$ and so that $\mex(r, N, P) \leq \frac{N^2}{2}(\mex(r,P) +\delta/2)$ for all $N\geq \frac{1}{\ep}$. Apply Lemma \ref{colored regularity lemma} to $\bigsqcup_{i\in[r]} G_i$ with regularity parameter $\ep$. Assume that $\{J, V_1,\cdots, V_k\}$ are the parts of the partition and each $V_i$ has size $l$, and let $R$ be the multicolor regularity multigraph. For each color $c$ let $R_c$ be the subgraph of $R$ of the edges of color $c$. We now show that for each $c$, $e(R_c)$ is large. At most $\ep l^2\binom{k}{2}$ $c$-colored edges can be between pairs which are not $\ep$-regular. At most $l^2 \binom{k}{2}d$ $c$-colored edges may be between pairs of density less than $d$. At most $k \binom{l}{2}$ $c$-colored edges may be within one of the parts $V_i$. At most $\ep n^2$ $c$-colored edges may be incident with the junk set $J$. Finally, for each edge in $R_c$ there are at most $l^2$ edges in $G_c$ between the corresponding parts. In total, we have
\[
e(G_c) \leq l^2 e(R_c) + \ep l^2\binom{k}{2}+l^2\binom{k}{2}d+k \binom{l}{2} + \ep n^2\leq l^2 e(R_c) + \ep \frac{l^2k^2}{2} + d\frac{l^2k^2}{2} + \ep \frac{l^2k^2}{2} + 2\ep \frac{l^2k^2}{2},
\]
where the last inequality uses $\ep \geq \frac{1}{k}$ and $n = kl+|J| \le kl+\ep n$. Therefore
\[
e(R_c) \geq \frac{k^2}{2} \left( \frac{e(G_c) - 4\ep - d}{\frac{1}{2}k^2l^2}\right) > \frac{k^2}{2}(\mex(r,P) + \delta/2),
\]
by the choice of $\ep$ and $d$. Since this inequality holds for all colors, and since $\ep$ was chosen so that $\mex(r, N, P) \leq \frac{N^2}{2}(\mex(r,P) +\delta/2)$ for all $N\geq \frac{1}{\ep}$, $P$ is a subgraph of $R$ and therefore $P^{(s)}$ is a subgraph of $R^{(s)}$. If $n$ (and thus $l$) is large enough, applying Lemma \ref{lem:embedding} shows that $P^{(s)}$ is a subgraph of $\bigsqcup_{i\in[r]} G_i$ and hence $G$ is a subgraph of $\bigsqcup_{i\in[r]} G_i$.
%We now choose parameters $\ep, d\in(0,1)$ with satisfying $\ep \ll d \ll \delta $ and also $d\ll 2^{-r}$ and we run the Regularity lemma (Lemma \ref{colored regularity lemma}) with the regularity parameter as $\ep$. This will ensure that for any subset of the colors $\mathcal{C}\subseteq [r]$ we have a regularity partition such that all but $\ep$-fraction of the pairs of clusters will be $\ep$-regular in the graph of the the $\mathcal{C}$-edges. \par Now, for every $\mathcal{C}\subseteq [r]$, we delete all the $\mathcal{C}$-colored edges that are either between non-$\ep$-regular pairs, incident on the junk set, contained inside one of the partitions, or are between pairs where the density of $\mathcal{C}$-colored edges is below $d$. For any color class $c\in[r]$, we delete at most $$\ep n
%^2 + \ep2^rn^2 + d2^{r}n^2$$ $c$-colored edges. Here, we need from our choice of parameters that
%$$\ep + \ep2^r + d2^{r} < \delta $$
%which can easily be attained. Hence, the modified graph still has that each color class $G_i'$ has density more than $\mex(r, G)$. Then by definition, the modified graph contains a copy of $G$. The modifications we have made to the graph ensures that the parts to which $G$ is embedded satisfies the hypotheses of Lemma \ref{lem:embedding}, and hence we may find a large enough blow-up of $P$, completing the proof. 
\end{proof}

\section{The upper bound}\label{sec:upper bound}
In this section, we prove the upper bound in Theorem \ref{thm:maintheorem}. Our main tool will be the following stability result of Füredi.

\begin{theorem}[Füredi, \cite{furedi}]\label{thm:furedi}
Let $G$ be a graph on $n$ vertices without a $K_{k+1}$, and let $t:=\textrm{ex}(n, K_{k+1})-e(G)$. Then, $G$ can be made $k$-partite by deleting at most $t$ edges. 
\end{theorem}

Let $G$ be an $r$-edge colored multigraph with $E_1,\ldots,E_r$ denoting the edge sets in colors $1,\ldots, r$ respectively. Let $\chi(G) = k+1$ and let $U_1,\ldots, U_{k+1}$ be the color classes of a proper $\chi$-coloring of $G$. Define an $r$-edge colored multigraph $P$ on vertex set $[k+1]$ where $xy$ is an edge of $P$ in color $i$ if and only if $U_x\cup U_y$ contains an edge in $E_i$. Then $P$ is a reduced graph of $G$. 

Let $m := M(P)$ be the maximum size of a matching in the subgraph made of the edges of multiplicity $1$ in $P$. Let $e_1,\ldots, e_m$ be a maximum matching of multiplicity $1$ edges in $P$. Define an $r$-edge colored multigraph $P'$ on vertex set $[k+1]$ that has a matching of size $m$ of single edges with the same colors as $e_1,\ldots, e_m$ and the remaining edges have multiplicity $r$ with $1$ edge of every color. Then $P$ is a subgraph of $P'$ and so by Proposition \ref{prop:erdosstone}, it suffices to show that 
\[
\mex(r, P)  \leq \mex(r, P') \leq 1 - \frac{1}{rk} - \frac{m}{9rk^2},
\]

To show this, assume that $H$ is an $n$-vertex $r$-edge colored multigraph and for $1\leq i\leq r$ let $H_i$ be the simple graph of the color $i$ edges of $H$. Assume that 
\[
e(H_i) > \left(1 - \frac{1}{rk} - \frac{m}{9rk^2}\right) \binom{n}{2},
\]
for all $i$. We will show that $H$ contains $P'$ as a subgraph. 

Let $R$ be the subgraph of $H$ of edges of multiplicity $r$. That is, $E(R) = \bigcap H_i$. Then 
\[
e(R) \geq \left(1 - \frac{1}{k} - \frac{m}{9k^2}\right)\binom{n}{2}.
\]

If $R$ contains $K_{k+1}$ as a subgraph, then $H$ contains $rK_{k+1}$ and hence $P'$, so we may assume that $R$ is $K_{k+1}$-free. Hence, by Tur\'an's theorem
\[
e(R) \leq \left(1 - \frac{1}{k}\right) \binom{n}{2}.
\]

By Theorem \ref{thm:furedi}, $R$ has a $k$-partite subgraph $R'$ satisfying
\[
e(R') \geq \left(1 - \frac{1}{k} - \frac{2m}{9k^2}\right) \binom{n}{2}.
\]
Assume that $V_1,\ldots, V_k$ are the partite sets of $R'$. Let $H_i'$ be the subgraph of $H_i$ consisting of all edges that have both endpoints in one of the partite sets. That is
\[
H_i' = \bigcup_{j=1}^k H_i[V_j].
\]
Then we have 
\[
e(H_i') \geq \left(1-\frac{1}{rk} - \frac{m}{9rk^2}\right) \binom{n}{2} - \left(1- \frac{1}{k}\right)\binom{n}{2}= \left(\frac{r-1}{r}\frac{1}{k} - \frac{m}{9rk^2}\right)\binom{n}{2},
\]
for each $i$.

Assume that $e_1,\ldots, e_m$ have colors $c_1,\ldots, c_m$ respectively. Choose $\pi \in S_k$ uniformly at random. For $1\leq j\leq m$, let $F_i$ be the graph in color $c_i$ induced by $V_{\pi(i)}$. That is, $F_i = H_{c_i}[V_{\pi(i)}]$. Let $F_i'$ be the subgraph of $F_i$ given by a maximum cut of $F_i$. Finally, consider the uncolored simple graph 
\[
H' = F_1'\cup \cdots \cup F_m' \cup R'.
\]

We claim that if $H'$ contains $K_{k+1}$, then $P'$ is a subgraph of $H$. To see this, assume that $K_{k+1}$ is a subgraph of $H'$. { We argue that the vertices corresponding to those of $K_{k+1}$ in $H'$ induce a $P'$ in $H$. Note that all edges of the $K_{k+1}$ that are not contained in a single part $V_i$ (i.e. edges that come from $R'$) have multiplicity $r$ in $H$. Further, the $K_{k+1}$ can draw at most two vertices from each $V_i$, as each $F_i'$ is a bipartite graph. Thus, the edges of $H'$ that come from the $F_i'$ form a matching. Hence, $H'$ corresponds to a multigraph in $H$ all of whose edges have multiplicity $r$ with $1$ edge of every color, except for a matching of size at most $m$ with single edges whose colors are a submultiset of $\{c_1, \ldots, c_m\}$. This structure contains a  $P'$.}

To complete the proof, we show that there is a choice of $\pi$ so that $H'$ contains $K_{k+1}$ by showing that $\mathbb{E}[e(H')] \geq \left(1 - \frac{1}{k}\right)\binom{n}{2}$ and applying Tur\'an's theorem.

Note that 
\[
\mathbb{E}[e(H')] = e(R') + \sum_{j=1}^m \mathbb{E}[e(F_i')] \geq \left(1 - \frac{1}{k} - \frac{2m}{9k^2}\right)\binom{n}{2} + \frac{1}{2}\sum_{j=1}^m \mathbb{E}[e(F_i)].
\]

Now, for each $j$, each edge in $H_{c_j}'$ has a $\frac{1}{k}$ chance of being in $F_j$. Therefore,
\[
\mathbb{E}[e(F_j)] \geq \frac{1}{k}\left(\frac{r-1}{r}\frac{1}{k} - \frac{m}{9rk^2}\right)\binom{n}{2},
\]
for $1\leq j\leq m$. Using $r \geq 2$ and $m< k$ gives
\[
\mathbb{E}[H'] > \left(1 - \frac{1}{k} - \frac{2m}{9k^2}\right)\binom{n}{2} + \frac{m}{2k}\left(\frac{1}{2}\frac{1}{k} - \frac{1}{18k}\right)\binom{n}{2} = \left(1 - \frac{1}{k}\right)\binom{n}{2}.
\]

\section{Discussion}
As we remarked in the introduction, it makes sense to study the $\mex(\cdot)$ parameter for specific families of graphs. When $r=2$, and when $G$ is bipartite, we showed that this problem reduces to a setting investigated in $\cite{TM}$. At any rate, the bound $\mex(G):=\mex(2, G)\leq 1/2$ is true, regardless of the coloring of $G$. If one wishes to see at which threshold a bicolored cycle will emerge in $R\cup B$ regardless of the coloring, the previous bound only leaves open the case of odd cycles. As stated in Theorem \ref{thm:cycle}, it turns out that the same upper bound of $1/2$ also holds for odd cycles. \begin{proof}[Proof of Theorem \ref{thm:cycle}]
Let $\textbf{C}$ be a bicolored odd cycle. If $\textbf{C}$ is monochromatic, then the result follows from the (uncolored) Erd\H{o}s-Stone-Simonovits theorem, and so we may assume that there is a path on two edges in $\mathbf{C}$ with colors RB. Since $\textbf{C}$ has an odd number of edges, it also contains a path on two edges colored either RR or BB. Assume that there is a RR path (if there is only a BB path, switch the colors in the remainder of the proof). Let $\textbf{T}_1$ be the bicolored triangle with one double edge, and the other two edges colored red. And let $\textbf{T}_2$ be the bicolored triangle with one double edge, one red edge, and one blue edge. It follows that $\textbf{C}$ can be embedded in a sufficiently large blow-up of either $\textbf{T}_1$ or $\textbf{T}_2$.
\par So the bound $\mex(\{\textbf{T}_1, \textbf{T}_2\})\leq 1/2$ would suffice to deduce the theorem, by Proposition \ref{prop:erdosstone}. We will actually show something stronger.
\begin{claim}
Let $\textbf{G}$ be a bicolored graph avoiding both $\textbf{T}_1$ and $\textbf{T}_2$. Then $e(\textbf{G})\leq \frac{n^2}{2}$
\end{claim}
Therefore, in any $\textbf{T}_1 / \textbf{T}_2$ avoiding $R\cup B$, one color class has at most $n^2/4$ edges, which implies that $\mex(\{\textbf{T}_1, \textbf{T}_2\})\leq 1/2$.

 To prove the claim, hence the theorem, we will proceed by induction. The base cases of $n=1$ and $n=2$ are clear. Let us fix an edge $\{x,y\}$ that is colored red (possibly also blue). If there wasn't one, there are at most $\binom{n}{2}$ edges (all blue) in the graph, and we are done.
\par By induction, $e(\textbf{G}\setminus \{x,y\})\leq \frac{(n-2)^2}{2}$. Further, for any $v\in G\setminus \{x,y\}$, $d(\{v\},\{x,y\})\leq 2$. Indeed, otherwise $v$ must send at least $3$ edges to $\{x,y\}$ and so they must include at least $1$ double edge. Then $\{v,x,y\}$ is either a $\textbf{T}_1$ or a $\textbf{T}_2$. In total,
$$ e(\textbf{G})\leq e(\textbf{G}\setminus \{x,y\}) + e(\textbf{G}\setminus \{x,y\}, \{x,y\}) + 2 \leq \frac{(n-2)^2}{2}+2(n-2)+2=\frac{n^2}{2}$$
\end{proof}

\par Note that our general Erd\H{o}s-Stone-Simonovits type theorem (Theorem \ref{thm:maintheorem}) cannot give the above optimal bounds in the above problem as $\chi(G)=3$ and $\mathcal{M}(G)\leq 1$ when $G$ is an odd cycle. Hence, we had to take a more direct approach. Although in general Theorem \ref{thm:maintheorem} is tight, the example we gave for tightness in Section \ref{sec:construction} featured a graph that is as dense as possible while having a particular chromatic number. Indeed, it was critical for our construction that the graph $H$ had clique number equal to its chromatic number, in fact, $H$ contained a large blow up of a clique of order $\chi(H)$. Outside this domain, we do not know if Theorem \ref{thm:maintheorem} is tight. It could be interesting to investigate $\mex(\cdot)$ in other sparser settings in an effort to determine if a more specific Erd\H{o}s-Stone-Simonovits theorem could be established in this setting, which would give a more complete answer to the question asked by Diwan and Mubayi. 
\par One natural direction is to study the average case. Say $G:=G(n, 1/2)$, the uniformly random graph, equipped with a uniformly random red-blue coloring of its edges. We know that $\chi(G)=(1/2+o(1))n/\log_2 n$ with high probability. It would be interesting to give bounds on $\mex(G)$ with high probability. It seems likely that the upper bound from our Theorem \ref{thm:maintheorem} would not be tight here. 
\par Another intriguing open problem is the one Diwan and Mubayi originally studied, namely determining $\mex(G)$ when $G:=K_n$ and $G$ is equipped with an arbitrary coloring of its edges. This problem seems quite hard, and restricting attention to almost all colorings of $K_n$ while aiming to determine $\mex(K_n)$ already seems to be a challenging question.
\par Finally, there are many natural variants or particular cases of this problem that have been studied and interesting open questions are pervasive. We conclude with a selection of these: forbidding rainbow triangles was considered in \cite{rainbowFlag} and \cite{rainbowMantel}; forbidding nonmonochromatic triangles was considered in \cite{DMM}; a survey of a more general problem with weights was given in \cite{thomasonSurvey}; an inverted version of the problem was asked in \cite{BC}; finding multiple rainbow cliques was studied in \cite{multipleRainbow}.

\bibliographystyle{plain}
\bibliography{bib}

\begin{thebibliography}{10}

\bibitem{rainbowFlag}
J{\'o}zsef Balogh, Ping Hu, Bernard Lidick{\`y}, Florian Pfender, Jan Volec,
  and Michael Young.
\newblock Rainbow triangles in three-colored graphs.
\newblock {\em Journal of Combinatorial Theory, Series B}, 126:83--113, 2017.

\bibitem{BC}
Joseph Briggs and Christopher Cox.
\newblock Inverting the {T}ur{\'a}n problem.
\newblock {\em Discrete Mathematics}, 342(7):1865--1884, 2019.

\bibitem{DMM}
Matt DeVos, Jessica McDonald, and Amanda Montejano.
\newblock Non-monochromatic triangles in a 2-edge-coloured graph.
\newblock {\em Electronic Journal of Combinatorics}, 26(P3.8):1--10, 2019.

\bibitem{Diestel}
Reinhard Diestel.
\newblock {\em Graph theory: Springer graduate text gtm 173}, volume 173.
\newblock Reinhard Diestel, 2012.

\bibitem{DM}
Ajit Diwan and Dhruv Mubayi.
\newblock Tur{\'a}n’s theorem with colors.
\newblock {\em preprint}, 2007.

\bibitem{multipleRainbow}
Stefan Ehard and Elena Mohr.
\newblock Rainbow triangles and cliques in edge-colored graphs.
\newblock {\em European Journal of Combinatorics}, 84:103037, 2020.

\bibitem{furedi}
Zoltan Füredi.
\newblock A proof of the stability of extremal graphs, simonovits' stability
  from szemerédi's regularity.
\newblock {\em Journal of Combinatorial Theory, Series B}, 115:66--71, 2015.

\bibitem{regularityDIMACS}
J{\'a}nos Koml{\'o}s and Mikl{\'o}s Simonovits.
\newblock Szemer{\'e}di's regularity lemma and its applications in graph
  theory.
\newblock 1996.

\bibitem{thomasonSurvey}
Edward Marchant and Andrew Thomason.
\newblock Extremal graphs and multigraphs with two weighted colours.
\newblock In {\em Fete of combinatorics and computer science}, pages 239--286.
  Springer, 2010.

\bibitem{TM}
Alp Müyesser and Michael Tait.
\newblock Tur\'an and ramsey-type results for unavoidable subgraphs.
\newblock {\em arXiv preprint arXiv:2004.07147}, 2020.

\bibitem{rainbowMantel}
Sebasti\'an González Hermosillo de la Maza Amanda Montejano Robert~\v{S}\'amal
  Ron~Aharoni, Matt~DeVos.
\newblock A rainbow version of {M}antel's theorem.
\newblock {\em Advances in Combinatorics}, 2:12pp., 2020.

\bibitem{regularity}
Endre Szemer{\'e}di.
\newblock Regular partitions of graphs.
\newblock Technical report, STANFORD UNIV CALIF DEPT OF COMPUTER SCIENCE, 1975.

\bibitem{thomason}
Andrew Thomason.
\newblock {\em Graphs, colours, weights and hereditary properties}, page
  333–364.
\newblock London Mathematical Society Lecture Note Series. Cambridge University
  Press, 2011.

\end{thebibliography}

\end{document}